\documentclass[12pt]{article}
\usepackage{amssymb,url, amsthm, amsmath, verbatim}

\usepackage{amssymb}

\newtheorem{Th}{Theorem}
\newtheorem{lem}{Lemma}
\newtheorem{defe}{Definition}
\newtheorem{rem}{Remark}

\newcommand{\eps}{\varepsilon}
\author{F.~V.~Petrov, D.~M.~Stolyarov, P.~B.~Zatitskiy}

\title{On embeddings of finite metric spaces in $l_\infty^n$}

\begin{document}

\maketitle

\begin{abstract}
We prove that for any given integer $c>0$ any metric space on $n$ points may be isometrically
embedded into $l_{\infty}^{n-c}$ provided $n$ is large enough.
\end{abstract}

Let $(X,\rho)$ be a metric space on $n$ points. Denote by $m(X)$ the minimal $k$ such that
$X$ may be isometrically embedded in $l_{\infty}^k$ and by $m(n)$ the maximum value of $m(X)$
for all metric spaces $X$ on $n$ points. Denote also $\alpha(X)=n-m(X)$, $\alpha(n)=n-m(n)$.
It is well known that $m(n)\leq n-1$. For example, one may fix a point
$x_0\in X$ and realize the $l_{\infty}^{n-1}$ as a vector space of functions on $X$, which vanish
in $x_0$, endorsed with max-norm. Then the map $x\rightarrow \rho(x,\cdot)-\rho(x_0,\cdot)$ defines
isometric embedding of $X$ into this space.  It is proved by D. Wolfe that if $n\geq 4$, then
$m(n)\leq n-2$ \cite{W}. Using Ramsey-type graphs with $n$ vertices
without 4-cycles and $k$-anticliques K. Ball has shown \cite{B} that $m(n)\geq n-k$. He refered to
Alon's \cite{A} explicit construction of such graphs with $k=O(n^{3/4})$, while the Spencer's combinatorial
argument \cite{S} allows us to get $k=O(n^{2/3}\cdot \log n)$, which gives $\alpha(n)=O(n^{2/3}\ln n)$.

Our main result is the following

\begin{Th}
$\lim \alpha(n)=+\infty$, i.e. that given $c>0$, $m(n)\leq n-c$ for large enough $n$.
\end{Th}

The idea is to use the well-known observation that if $Y\subset X$, then $m(X)\leq m(Y)+|Y\setminus X|$, hence
$\alpha(X)\geq \alpha(Y)$. Then we may try to find appropriate subset in $X$ using Ramsey theorem and work with this subset
instead $X$.

We start with the following technical

\begin{defe} A finite metric space $X$ is called \textsl{generic}, if the distances between its
points are linearly independent over $\mathbb{Q}$.
\end{defe}

and the standard

\begin{lem} If $m(X)\leq N$ for any generic metric space $X$ on $n$ points, then $m(n)\leq n$, i.e. the same inequality
holds for any metric space $X$ on $n$ points.
\end{lem}

\begin{proof}
Fix arbitrary metric space $(X,\rho)$ on $n$ points. Our aim is to prove that $M(X)\leq N$. Fix $\eps>0$. Change a metric
on $X$ by adding to each distance some number from $[\eps,2\eps]$ so that the new metric $\rho_{\eps}$ is generic.
It is clearly possible: just change the distances step by step, and on each step you have only countably many forbidden
changes. Then $(X,\rho_{\eps})$ may be embedded isometrically in $l_{\infty}^{N}$, without loss of generality a point $x_0\in X$
maps to 0. Then let $\eps$ tend to 0, consider the convergent subsequence of embeddings and its limit is
isometric embedding of $(X,\rho)$ in $l_{\infty}^{n-1}$.
\end{proof}

Hereafter we consider only generic metric spaces.

Assume that $X$ is embedded in $l_{\infty}^N$. For any $i=1,\,2,\,\dots,\,N$ consider the value of $i$-th coordinate
as a function $\phi_i$ defined on $X$. Note that $\phi_i$'s are 1-Lipschitz functions (that is, $|\phi_i(x)-\phi_i(y)|\leq \rho(x,y)$),
and for any $x,y\in X$ there exists $i$ such that equality $|\phi_i(x)-\phi_i(y)|= \rho(x,y)$ holds. These two
conditions mean nothing but that our map is isometry on $X$. For any 1-Lipschitz function $f$ on $X$ define
its (oriented) graph, vertices of which are points of $X$, and edge $a\rightarrow b$ is drawn iff $f(a)-f(b)=\rho(a,b)$.
Then we realize that $m(X)$ is the minimal number of graphs of 1-Lipschitz functions on $X$, which cover all the edges of
complete (non-oriented) graph on $X$. Recall the following well-known

\begin{lem} If $Y\subset X$, then any 1-Lipschitz function $f$ on $Y$ may be extended to 1-Lipschitz function on $X$
\end{lem}

\begin{proof} It suffices to consider the case $|X\setminus Y|=1$, $X=Y\cup\{x_0\}$ (and then use induction).
Define $f(x_0)$ as $f(x_0)=\max_{y\in Y} (f(y)-\rho(y,x_0))$. Let this maximum achieve in $z\in Y$. By definition, $f$
satisfies $f(x_0)-f(y)\geq -\rho(y,x_0)$ for any $y\in Y$. So, to check that $f$ is 1-Lipschitz on $X$, we need to check
that $f(x_0)-f(y)\leq \rho(y,x_0)$ for any $y\in Y$. We have $f(x_0)-f(y)=f(z)-f(y)-\rho(z,x_0)\leq \rho(z,y)-\rho(z,x_0)\leq \rho(y,x_0)$
and we are done.
\end{proof}

So, if some graphs of 1-Lipschits functions of $Y\subset X$ cover complete graph on $Y$, we may extend these functions to $X$
so that they are still 1-Lipschitz and additionally consider the functions $\rho(z,\cdot)$ for all $z\in X\setminus Y$.
The graphs of these functions cover all the edges having at least one endpoint not in $Y$. Therefore we get $m(X)\leq m(Y)+|X\setminus Y|$.
hence if we find a subset $Y\subset X$ such that $m(Y)\leq |Y|-c$, then $m(X)\leq |X|-c=n-c$ aswell.


Let $T$ be a tree on $X$. Orient its edges so that there is no path of length 2 $a\rightarrow b\rightarrow c$.
There are two ways to do it, choose any. Then define a function on $X$ so that $f(a)-f(b)=\rho(a,b)$ for any edge $a\rightarrow b$
of our graph. It may be done uniquely to adding a constant function. Again, choose any.
We need the following straightforward

\begin{lem} If $T$ is a tree such that any two vertices are joined by a path of at most 4 edges,
then such a function is 1-Lipschitz iff for any path $a-b-c-d$ in a (non-oriented) tree $T$ we have
$\rho(a,d)+\rho(b,c)\geq \rho(a,b)+\rho(c,d)$.
\end{lem}

\begin{proof}
Consider any $x,y\in X$ and check whether $|f(x)-f(y)|\leq \rho(x,y)$.
If $x$ and $y$ are joined by edge, clearly the equality holds. If they are joined by the path
$x-z-y$, then $|f(x)-f(y)|=|\rho(x,z)-\rho(z,y)|\leq \rho(x,y)$, which is still ok.
If they are joined by the path $x-z-w-y$, then $|f(x)-f(y)|=|\rho(x,z)-\rho(z,w)+\rho(w,y)|$.
Inequality $\rho(x,z)-\rho(z,w)+\rho(w,y)\geq -\rho(x,y)$ always holds and follows from triangle
inequality. Inequality $\rho(x,z)-\rho(z,w)+\rho(w,y)\leq \rho(x,y)$ is exactly the condition which we require
in Lemma. Finally, if $x$ and $y$ are joined by the path $x-z-t-w-y$, then
$f(x)-f(y)=\pm(\rho(x,z)-\rho(z,t)+\rho(t,w)-\rho(w,y))$. For checking, say,
$\rho(x,z)-\rho(z,t)+\rho(t,w)-\rho(w,y)\leq \rho(x,y)$ use $\rho(x,z)-\rho(z,t)+\rho(t,w)\leq \rho(x,w)$, which is
again the requirement of Lemma.
\end{proof}

Analogous statement holds for all trees, but here we need only trees of diameter at most 4.

Now we are ready for applying Ramsey theorem.

It is convenient to think that the points
of our space are reals $x_1<x_2<\dots<x_n$. For any indexes $1\leq a<b<c<d\leq n$ consider the four points
$x_a$, $x_b$, $x_c$, $x_d$ and the following sums: $R_1=\rho(x_a,x_b)+\rho(x_c,x_d)$, $R_2=\rho(x_a,x_c)+\rho(x_b,x_d)$,
$R_3=\rho(x_a,x_d)+\rho(x_b,x_c)$. These sums are different reals (since $X$ is generic), and there are six
possible ways to rearrange them. So, take six colours labelled by elements of the symmetric group $S_3$ and colour
the quadruple $(a,b,c,d)$ in depend on the arrangement of $R_1$, $R_2$, $R_3$ (for example, colour $(a,b,c,d)$ in color
231, if $R_2>R_3>R_1$).

Fix positive integer $k$ and note that if $n$ is large enough, then by Ramsey theorem there exist $k$ indexes such that all quadruples
formed by these indexes have the same colour. We call the corresponding subspace of $X$ monochromatic of corresponding
colour.

So, it suffices to consider only monochromatic spaces. Call our space $(X,\rho)$, $|X|=n$ again, it is now generic and monochromatic.

We have six cases, which correspond to six permutations of 1,2,3. The first thing which we note is that two of them are impossible
if $n\ge 5$.

\begin{lem} There are no monochromatic generic metric spaces on 5 points of colours 213 or 312.
\end{lem}

\begin{proof} Assume that $(X,\rho)$ is a generic metric space on 5 points, monochromatic of colour 213.
Let $x_1<x_2<x_3<x_4<x_5$ be the vertices of $X$. Then
\begin{align*}
\rho(x_1,x_2)+\rho(x_3,x_4)&>\rho(x_1,x_4)+\rho(x_2,x_3)\cr
\rho(x_2,x_3)+\rho(x_4,x_5)&>\rho(x_2,x_5)+\rho(x_3,x_4)\cr
\rho(x_1,x_4)+\rho(x_2,x_5)&>\rho(x_1,x_2)+\rho(x_4,x_5).
\end{align*}
Sum up to get a contradiction. The case of colour 312 is analagous, just change the sign in 3 above inequalities.

\end{proof}

Monochromatic spaces of other four colours and arbitrary cardinality do exist, so we have to consider
four cases. We use different approaches in all the four cases.

But before passing to separate cases, make the following general note about trees of diameter at most 4,
which are graphs of 1-Lipschitz function (call such trees \textsl{admissible}).
Each such tree has a \textsl{center}, some vertices, joined by edge with a center, which we call
\textit{main vertices}, and other, \textsl{peripheric}, vertices, each of them is joined with one of main vertices.

Note that if $Y\subset X$ and $T$ is admissible tree on $Y$, then $T$ may be extended to an admissible tree on $X$
with the same set of main vertices.
Indeed, it suffices to consider $X=Y\cup \{x_0\}$. Join $x_0$ with such a main vertex $a$  that $\rho(x_0,a)-\rho(a,o)$
takes minimal value ($o$ is center of $T$). Moreover, this expending is unique and so the admissible tree is uniquely determined
by its center and main vertices. Denote the admissible tree with the center $o$ and admissible vertices
$a_1,a_2,\dots,a_m$ as $T(o;a_1,a_2,\dots,a_m)$.

Note also that the covering a monochromatic space by graphs of 1-Lipschitz functions does depend only on the colour
of our monochromatic space and does not depend on metric structure in any other way. So, we
may forget about the metric structure and define a unique monochromatic structure of given colour on any
finite set of reals.

\textbf{Case 321.}

Let's prove that $k$ trees are enough to cover 321-monochromatic space $X$ with $2k$ vertices.
Set $X=\{-k,-(k-1),\dots,-1,1,\dots,k\}$. For any $i=1,2,\dots,k$ consider the tree
$T(-i;-k,-(k-1),\dots,-i-1,1,2,\dots,i)=T(i;k,(k-1),\dots,i+1,-1,-2,\dots,-i)$. It's straightforward to check that these trees
are admissible and cover $X$. Note that they all have indeed diameter 3, not 4.

\textbf{Case 132.}

We use the following

\begin{lem} Let $Y$ be a 132-monochromatic space, $|Y|=n$ and $m$ admissible trees
cover the complete graph on $Y$ without some $k$ edges. Then we may add 2 vertices
(i.e. 2 reals) to $Y$ and for new space $X$ add 2 admissible trees so that
now the trees (2 added and extensions of $m$ old) cover complete graph on $X$ without $k-1$ edges.
\end{lem}

In other words, we add two additional vertices and two additional trees, but kill
one edge.

\begin{proof} Assume that the edge $a-b$, $a<b$, is not covered by our trees. Add two vertices
$a+\eps$, $b-\eps$ and two trees $T(a;a+\eps,b)$ and $T(b;a+\eps,b-\eps)$. The check is straightforward.
\end{proof}

Then, using this Lemma, we start with $|Y|=c$, $m=0$ and $k=c(c-1)/2$. Apply the Lemma $k$ times
and get $X$ such that $|X|=c+2k$ and $X$ is covered by $2k$ trees, as desired.

\textbf{Case 123.}

We prove that for given integer $c>0$ there exists large $N$ such that the 123-monochromatic
space on $N$ vertices may be covered by $N-c$ admissible trees. It is exactly what we need.

Use induction on $c$. For $c=0$ take $N=1$.

Now we show that if $N$ is good for $c$, then $2N+3$ is good for $c+1$.
Take vertices $-1,0,1,\dots,2N+1$. Take the trees
$T(0;-1,2)$, $T(1;0,2)$ and $T(i+N+1;i,i+1)$ for $i=1,\,2,\,\dots,N$.
So, we use $N+2$ trees and it is straightforward that they cover all edges
having at least one endpoint in $\{-1,0,1,\dots,N+1\}$. All the other edges may be covered by
$N-c$ trees by induction proposition. So, we use $2N+2-c=(2N+3)-(c+1)$ trees to cover all the edges.

\textbf{Case 231}

We take a 231-monochromatic space on $4n+1$ vertices $\{0,1,\dots,4n\}$ and cover it by $3n$ trees.
It is enough for our purpose for $n>c$.

For $i=1,2,\dots,2n$ take trees $T(0;i,4n+1-i)$. They cover all the edges except edges $a-b$
for $1\leq a\leq 2n<b\leq 4n+1-a$. So, we have to cover the bipartite graph (with $n$ vertices in each part) formed by such edges
by $n$ trees. Note that the tree $T(1;2n+1,2n+2,\dots,4n)$ covers all its edge from
vertices $1,2n,2n+1,4n$. After taking this tree we have analogous graph with $n-1$ vertices in each part
by $n-1$ trees. This is made by repeating this procedure (or by induction).

So, all the cases are considered and the Theorem 1 is proved.

\begin{rem} Another natural question is to study the minimal $k(n)$ such that any metric space on $n$
point may be embedded in some $k$-dimensional Banach space, not necessary in $l_{\infty}^k$. Clearly,
$k(n)\leq m(n)$, and we do not know, whether the equality always holds or not. The probably
best known lower bound for $k(n)$ is $2n/3$ \cite{AD}.
\end{rem}

\end{document}